\documentclass[11pt]{article}

\makeindex
\usepackage[all]{xy}
\usepackage{graphicx}

\usepackage{amsmath}
\usepackage{amsfonts}
\usepackage{amssymb}
\usepackage{amsthm}

\DeclareMathOperator{\im}{im}
\DeclareMathOperator{\coker}{coker}
\newcommand{\isom}{\cong}

\newcommand{\Z}{{\bf{Z}}}

\newcommand{\Q}{{\bf{Q}}}
\newcommand{\Qbar}{{\overline{\Q}}}
\newcommand{\R}{{\bf{R}}}
\newcommand{\C}{{\bf{C}}}

\newcommand{\T}{{\bf{T}}}

\newcommand{\m}{{\mathfrak{m}}}
\newcommand{\OO}{\mathcal{O}}
\newcommand{\NN}{\mathcal{N}}
\newcommand{\JJ}{\mathcal{J}}

\newcommand{\Aa}{\mathcal{A}}
\newcommand{\GG}{\mathcal{G}}
\newcommand{\CA}{{C_{\scriptscriptstyle{A}}}}

\newcommand{\noi}{{\noindent}}

\newcommand{\Mid}{|} 

\newcommand{\miD}{|}

\newcommand{\divs}{\!\mid\!}
\newcommand{\ndiv}{\!\nmid\!}
\newcommand{\tensor}{\otimes}

\newcommand{\ra}{{\rightarrow}}

 1
\DeclareFontEncoding{OT2}{}{} 
  \newcommand{\textcyr}[1]{%
    {\fontencoding{OT2}\fontfamily{wncyr}\fontseries{m}\fontshape{n}%
     \selectfont #1}}
\newcommand{\Sha}{{\mbox{\textcyr{Sh}}}} 

\newcommand{\mm}{{\mathfrak{m}}}

\newcommand{\ord}{{\rm ord}}

\newcommand{\annT}{{{\rm Ann_\T}}}
\newcommand{\Eis}{{\Im}}

\newcommand{\nJ}{{\eta_{\scriptscriptstyle{J}}}}

\newcommand{\comment}[1]{}

\DeclareMathOperator{\tor}{tor}

\newtheorem{lem}{Lemma}[section]
\newtheorem{cor}[lem]{Corollary}
\newtheorem{prop}[lem]{Proposition}
\newtheorem{conj}[lem]{Conjecture}
\newtheorem{thm}[lem]{Theorem}

\theoremstyle{definition}

\newtheorem{que}[lem]{Question}

\newcommand{\thetitle}
{Mod-$p$ reducibility, the torsion subgroup, and the Shafarevich-Tate group}

\begin{document}
\parindent=2em

\title{\thetitle}
\author{Amod Agashe
\footnote{This material is based upon work supported by the National Science 
Foundation under Grant No. 0603668.}}
\maketitle

\begin{abstract}
Let $E$ be an optimal elliptic curve over~$\Q$ of prime conductor~$N$.
We show that if for an odd prime~$p$, the mod~$p$ representation
associated to~$E$ is reducible
(in particular, if $p$ divides the order of the torsion
subgroup of~$E(\Q)$), then 
the $p$-primary component of the Shafarevich-Tate
group of~$E$ is trivial. 
We also state a related result for more general abelian
subvarieties of~$J_0(N)$  and 
mention what to expect
if $N$ is not prime.
\end{abstract}

\section{Introduction and results}

Let $N$ be a positive integer. Let
$X_0(N)$ be the modular curve over~$\Q$
associated to~$\Gamma_0(N)$, and
let $J=J_0(N)$ denote the Jacobian of~$X_0(N)$, which is
an abelian variety over~$\Q$. 
Let $\T$ denote the Hecke algebra, which is 
the subring of endomorphisms of~$J_0(N)$
generated by the Hecke operators (usually denoted~$T_\ell$
for $\ell \ndiv N$ and $U_p$ for $p\divs N$). 
If $f$ is a newform of weight~$2$ on~$\Gamma_0(N)$, then 
let $I_f = \annT f$ and let $A_f$ denote the associated {\em newform
quotient} $J/I_f J$, which is
an abelian variety over~$\Q$. 
If the newform~$f$ has integer Fourier coefficients, then
$A_f$ is an elliptic curve, 
and we denote it by~$E$.
It is called the {\em optimal elliptic curve}
associated to~$f$ and its conductor is~$N$ (which we may also call
the {\em level} of~$E$).
If $A$ is an abelian variety over a field~$F$, then
as usual, $\Sha(A/F)$ denotes the Shafarevich-Tate group of~$A$ over~$F$.

\begin{thm} \label{thm:mainec}
Let $E$ be an optimal elliptic curve of prime conductor~$N$.
If $p$ is an odd prime such that $E[p]$
is reducible as a ${\rm Gal}(\Qbar/\Q)$ representation
over~$\Z/p\Z$, then the $p$-primary component of~$\Sha(E)$ 
is trivial. In particular, if an odd prime~$p$ 
divides~$| E(\Q)_{\rm tor} |$, then $p$ does not divide~$|\Sha(E)|$,
assuming that $\Sha(E)$ is finite.
\end{thm}

We shall soon deduce this theorem from 
a more general result involving abelian subvarieties
of~$J_0(N)$, which we discuss next. 
The Eisenstein ideal~$\Eis$ of~$\T$ is the ideal  generated
by $1 + W_N$ and $1+\ell - T_\ell$ for all primes $\ell \ndiv N$,
where $W_N$ denotes the Atkin-Lehner involution.
Prop.~II.14.2 in~\cite{mazur:eis} implies:

\begin{prop}[Mazur] \label{prop:mazur}
Let $\m$ be a maximal ideal of~$\T$ with odd residue characteristic.
If $J_0(N)[\m]$ is reducible as a ${\rm Gal}(\Qbar/\Q)$ representation
over~$\T/\m$, then $\Eis \subseteq \m$.
\end{prop}

\begin{thm} \label{thm:main}
Suppose that $N$ is prime.
Let $A$ be an abelian subvariety
of~$J_0(N)$ that is stable under the action of the Hecke algebra~$\T$.
If $\m$ is a  maximal ideal of~$\T$
containing~$\Eis$ and having odd residue characteristic~$p$, 
then $\Sha(A)[\m]=0$ (equivalently
$\Sha(A)_\m \tensor_{\T_p} \T_\m = 0$).
\end{thm}

We shall prove this theorem in Section~\ref{proof}.
For now, we just remark that our proof follows that 
of~\cite[III.3.6]{mazur:eis} (which proves 
proves the theorem above for the case where $A = J_0(N)$), 
with an extra input coming 
from~\cite{emerton:optimal}.

\begin{cor} \label{cor:main}
Suppose that $N$ is prime.
Let $A$ be an abelian subvariety
of~$J_0(N)$ that is stable under the action of the Hecke algebra~$\T$.
If $\m$ is maximal ideal of~$\T$
with odd residue characteristic  such that $A[\m]$
is reducible as a ${\rm Gal}(\Qbar/\Q)$ representation
over~$\T/\m$, then $\Sha(A)[\m]=0$ (equivalently
$\Sha(A)_\m \tensor_{\T_p} \T_\m = 0$).
\end{cor}
\begin{proof}
Since $A[\m]$ is a sub-representation of~$J_0(N)[\m]$,
we see that $J_0(N)[\m]$ is reducible as well.
Then by Proposition~\ref{prop:mazur}, we see that $\Eis \subseteq \m$.
The corollary now follows from Theorem~\ref{thm:main}.
\end{proof}

We remark that if $A$ is as in the corollary above, then
we do not expect that 
a statement analogous to the first conclusion of 
Theorem~\ref{thm:mainec} holds if $A$ is not an elliptic curve,
i.e., it may not be true that if $p$ is an odd prime such
that $A[p]$
is reducible as a ${\rm Gal}(\Qbar/\Q)$ representation
over~$\Z/p \Z$, then $\Sha(A)[p]=0$, although we do not
know a counterexample.

\begin{proof}[Proof of Theorem~\ref{thm:mainec}]

Let $f$ denote the newform corresponding to~$E$ and
let $\sum_{n > 0} a_n(f) q^n$ be its Fourier expansion.
For all primes~$\ell \ndiv N$, 
$T_\ell$ acts as multiplication by~$a_\ell(f)$ on~$E$, 
and $U_p$ acts as multiplication by~$a_p(f)$ for all primes $p \divs N$.
Also, $T_1$ acts as the identity. Hence we get a surjection
$\phi: \T \ra \Z$, whose kernel is~$I_f$. 
The kernel of the composite $\T \stackrel{\phi}{\ra} \Z \ra \Z /p \Z$
is then $(p, I_f)$, the ideal of~$\T$ generated by~$p$ and~$I_f$.
Since the composite is also surjective, $\T / (p, I_f) \isom \Z/p\Z$
and so $(p, I_f)$ is a maximal
ideal of~$\T$, which we denote by~$\m$. 
Then $\m$ annihilates~$E[p]$, and so $E[p] \subseteq E[\m]$.
Conversely, since $p \in \m$, $E[\m] \subseteq E[p]$.
Hence $E[\m] = E[p]$.
By hypothesis, 
$E[p]$ is reducible over~$\Z/ p\Z$. Hence by the discussion above,
$E[\m]$ is reducible over~$\T /\m$. 
Then by Corollary~\ref{cor:main}, $\Sha(E)[\m] = 0$.
Now $\Sha(E)[p]$ is annihilated by~$p$ and~$I_f$, hence by~$\m$.
Thus $\Sha(E)[p] \subseteq \Sha(E)[\m]$, and so 
$\Sha(E)[p] = 0$.
If the $p$-primary part of~$\Sha(E)$ were non-trivial, then
so would~$\Sha(E)[p]$. 
Hence the
$p$-primary part of~$\Sha(E)$ is trivial, as was to be shown.
\end{proof}

In Section~\ref{sec:euler}, we mention the relevance of Theorem~\ref{thm:mainec}
from the point of view of
the second part of the Birch and Swinnerton-Dyer conjecture
and the results towards the conjecture coming from the theory of Euler systems.
In Section~\ref{sec:nonprime}, we discuss what one may expect
if the level is not prime. Finally, in Section~\ref{proof},
we give the proof of Theorem~\ref{thm:main}.
\vspace{0.1 in}

\noi {\em Acknowledgements:} We are grateful to E.~Aldrovandi and
F.~Calegari for answering some questions related to this article,
and to W.~Stein for helping us use the mathematical software sage.
Some of the computations were done on sage.math.washington.edu, which
is supported by National Science Foundation Grant No. DMS-0821725.

\section{Applications} \label{sec:euler}

We first recall the second part of the Birch and Swinnerton-Dyer conjecture
(see, e.g.,\cite[\S16]{silverman:aec} for details).
Let $L_E(s)$ denote the
{\em $L$-function} of $E$ and let 
$r$ denote the order of vanishing of $L_E(s)$ at $s=1$; it is
called the {\it analytic rank} of~$E$.
Let $c_p(A)$ denote the order of the arithmetic component group 
of the special fiber at the prime~$p$ of the N\'{e}ron 
model of~$E$ (so $c_p(E)=1$ for almost every prime). Let
$\Omega_E$ denote the volume of~$E({\R})$ calculated using
a generator
of the group of invariant differentials on the N\'{e}ron 
model of~$E$ and let $R_E$ denote the regulator of~$E$. 

The {\it second part of the Birch and Swinnerton-Dyer conjecture}
asserts the formula:
\begin{eqnarray} \label{bsd}
\frac{\lim_{s \ra 1} \{ (s -1 )^{-r} L_E(s) \}}{\Omega_E R_E} 
= \frac {\Mid \Sha_{E} \miD \cdot \prod_p c_p(E)}
{{\Mid E(\Q)_{\tor} \miD}^2}\ .
\end{eqnarray}

If $E$ is an elliptic curve and $N$ is prime, 
then by~\cite[Theorem~B]{emerton:optimal}, 
$c_N(E) = {\Mid E(\Q)_{\tor} \miD}$, so there is some cancellation
on the right side of~(\ref{bsd}). 
Our result shows that however, in this case, 
there is no cancellation between the odd parts of
$\Mid \Sha_{E} \miD$ and ${\Mid E(\Q)_{\tor} \miD}$; in particular,
up to a power of~$2$,
the numerator of the right side of~(\ref{bsd})
is precisely  $\Mid \Sha_{E} \miD$ 
and the denominator
is ${\Mid E(\Q)_{\tor} \miD}$.

\comment{
The significance of the project above is as follows: as mentioned
before, the 
theory of Euler systems can be used to show that
the actual order of the  Shafarevich-Tate group divides the order
predicted by the BSD conjecture, but staying away from certain primes~$p$,
which include the primes~$p$ 
for which the $p$-torsion subgroup
of the elliptic curve is reducible as a module over the absolute Galois group
of~$\Q$. Thus our project complements the theory 
of Euler systems and is a first step in trying to fill the gap 
in the results coming from the Euler system machinery at primes where
the Galois representation is reducible. 
}

We return to the case where $N$ is arbitrary (not necessarily prime).
Suppose $E$ is an elliptic curve whose analytic rank is zero or one.
Then by results of~\cite{kollog:finiteness},
$\Sha_E$ is finite, and 
moreover, one can use the theory of
Euler systems to bound~$\ord_p(|\Sha_E|)$ for a prime~$p$
from above in terms of the order conjectured by
formula~(\ref{bsd}) under certain hypotheses on~$p$.
As far as the author is aware, the hypotheses on the prime~$p$
include either the hypothesis that
the image of ${\rm Gal}(\Qbar/\Q)$ acting on~$E[p]$ is 
isomorphic to~${\rm GL}_2(\Z/p\Z)$  or the weaker hypothesis that 
the ${\rm Gal}(\Qbar/\Q)$-representation $E[p]$ 
is irreducible.
Our result shows that the latter hypothesis is redundant when $N$ is prime,
as we now explain in more detail.

Let $K$ be a quadratic imaginary field of discriminant not equal
to~$-3$ or~$-4$, and  such that
all primes dividing~$N$
split in~$K$. Choose an ideal~$\NN$ of the ring of integers~$\OO_K$
of~$K$ such that $\OO_K/\NN \isom \Z/N \Z$. Then the complex tori
$\C/\OO_K$ and~$\C/\NN^{-1}$ define elliptic curves related by 
a cyclic $N$-isogeny, and thus give a complex valued point~$x$ of~$X_0(N)$.
This point, called a Heegner point, is defined over the
Hilbert class field~$H$ of~$K$.
Let $P \in J(K)$ be the 
class of the divisor 
$\sum_{\sigma \in {\rm Gal}(H/K)} ((x) - (\infty))^\sigma$, where
$H$ is the Hilbert class field of~$K$. 

By~\cite{bump-friedberg-hoffstein}, we may choose K so 
that $L(E/K,s)$ vanishes to order one at $s=1$.
Hence, by~\cite[{\S}V.2:(2.1)]{gross-zagier}, $\pi(P)$ has 
infinite order, and by work of Kolyvagin, $E(K)$ has rank one and
the order of the Shafarevich-Tate group~$\Sha(E/K)$ of~$E$ over~$K$ 
is finite
(e.g., see~\cite[Thm.~A]{kolyvagin:euler} or~\cite[Thm.~1.3]{gross:kolyvagin}).
In particular, the index $[E(K):\Z\pi(P)]$
is finite.
By~\cite[{\S}V.2:(2.2)]{gross-zagier}
(or see~\cite[Conj.~1.2]{gross:kolyvagin}), the second part of
the Birch and Swinnerton-Dyer conjecture becomes:
\begin{conj}[Birch and Swinnerton-Dyer, Gross-Zagier] \label{conj:bsd}
\begin{eqnarray} \label{gzformula}
|E(K)/ \Z \pi(P)| 
= c_{\scriptscriptstyle E} \cdot \prod_{\ell | N} c_\ell(E) \cdot 
\sqrt{\Mid \Sha(E/K) \miD},
\end{eqnarray}
where $c_{\scriptscriptstyle E}$ is the Manin constant of~$E$.
\end{conj}

Note that the Manin constant~$c_{\scriptscriptstyle E}$
is conjectured to be one, and
one knows that if $p$ is a prime such that $p^2 \nmid 2 N$, 
then $p$ does not divide~$c_{\scriptscriptstyle E}$
(by~\cite[Cor.~4.1]{mazur:rational} 
and~\cite[Thm.~A]{abbes-ullmo}).

The following is~\cite[Cor.~1.5]{jetchev}:
\begin{prop}[Jetchev] \label{prop:jetchev}
Suppose that 
$p$ is a prime such that $p \ndiv 2 N$,
the image ${\rm Gal}(\Qbar/\Q)$ acting on~$E[p]$ is isomorphic to~${\rm GL}_2(\Z/p\Z)$ ,
and $p$ divides at most
one~$c_\ell(E)$. 
Then 
$$\ord_p (\Mid \Sha(E/K) \miD) \leq \ord_p (\Mid \Sha(E/K) \miD_{\rm an}).$$
\end{prop}

We also have:

\begin{prop}[Cha] \label{prop:cha}
Suppose that $p$ is an odd prime such that $p$ does not divide
the discriminant of~$K$, $p^2 \nmid N$, and 
the ${\rm Gal}(\Qbar/\Q)$-representation $E[p]$ 
is irreducible. Then
$$\ord_p (\Mid \Sha(E/K) \miD) \leq 
\ord_p(\Mid \Sha(E/K) \miD_{\rm an}) + 
\ord_p\bigg(\frac{\prod_{\ell | N} c_\ell(E)}
{\Mid E(K)_{\rm tor} \miD}\bigg).$$
\end{prop}
\begin{proof}
This follows from Theorem~21 of~\cite{cha}, in view of
the formula in Conjecture~4 of loc. cit. and the fact that
under the hypotheses, the Manin constant of~$E$
is one by~\cite[Cor.~4.1]{mazur:rational}. 
\end{proof}

The two results above are typical results coming from the theory
of Eulers systems.
The inflation-restriction sequence shows that 
the natural map $\Sha(E/\Q) \ra \Sha(E/K)$ has kernel 
a finite group of order a power of~$2$, hence
the above results give bounds on~$\Mid \Sha(E/\Q) \miD$
as well.
Note that  the hypotheses of Proposition~\ref{prop:jetchev}
implies that 
the ${\rm Gal}(\Qbar/\Q)$-representation $E[p]$ 
is irreducible.
Our result shows that if $N$ is prime, then 
the bound on~$\Mid \Sha(E/\Q) \miD$ obtained from
Proposition~\ref{prop:jetchev}  holds even
if the ${\rm Gal}(\Qbar/\Q)$-representation $E[p]$ 
is reducible and  
the bound on~$\Mid \Sha(E/\Q) \miD$ obtained from
Proposition~\ref{prop:cha} holds even without
the hypothesis 
that the ${\rm Gal}(\Qbar/\Q)$-representation $E[p]$ 
is irreducible. 

Finally, we remark that Theorem~\ref{thm:main}, which applies
not just to elliptic curves, but more generally to
abelian subvarieties of~$J_0(N)$ that are
stable under the action of the Hecke algebra, may have
an application to a potential equivariant Tamagawa number
conjecture where the quantities in the Birch and Swinnerton-Dyer
conjectural formula (analogous to~(\ref{bsd}) for 
such abelian varieties) are treated as modules over
the Hecke algebra. Also, this more general theorem
motivates the need to come up with 
Euler systems type results for $\m$-primary parts of the Shafarevich-Tate
group for maximal ideals~$\m$ of the Hecke algebra
(as opposed to $p$-primary parts for some prime~$p$).

\section{Non-prime conductors} \label{sec:nonprime}
One may wonder what happens if the hypothesis that $N$
is prime is dropped in any of our results above.
We searched through Cremona's database~\cite{cremona:onlinetables}
of elliptic curves
of conductor up to~130,000 using the mathematical software sage
to find all curves~$E$ such that
there
is an odd prime~$p$ such that
$E[p]$ is
reducible, and $p$ divides the Birch and Swinnerton-Dyer
 conjectural order of~$\Sha(E/\Q)$.
We found several examples, some of which are listed in Table~\ref{table1}.
In the table, the first column gives Cremona's label for the elliptic
curve, the second column is the prime factorization of the
conductor, the third column is
the  rank of~$E(\Q)$, the fourth column lists 
all {\em odd} primes~$p$ such that $E[p]$ is
reducible, the fifth column is the order of the odd part
of~$E(\Q)_{\rm tor}$,
the sixth column is the odd part of~$\prod_p c_p(E)$, and the last column
is the odd part of the Birch and Swinnerton-Dyer conjectural order of~$\Sha(E/\Q)$.

\begin{table}\caption{\label{table1}}
\begin{center}
\begin{tabular}{|l|l|l|l|l|l|l|}\hline
E & N & r & $p : E[p]$ is& $|E(\Q)_{\rm tor}|$ & $\prod_p c_p(E)$ & $|\Sha_{\rm an}|$\\
   &  &     & reducible    &    &   &  \\\hline
2366d1 & $2\cdot 7\cdot 13^2$ & 0 & 3 &  3 & 3 & 9 \\\hline
5054c1  & $2\cdot 7\cdot 19^2$  & 0  &  3  & 1  & 1  & 9\\\hline
46683w1  & $3^3\cdot 7\cdot 13\cdot 19$  & 1  & 3  & 3  & 1  & 9\\\hline
10621c1  & $13\cdot 19\cdot 43$  & 0  & 3  &  3  &  1  &  9\\\hline
33825o1  & $3 \cdot 5^2 \cdot 11 \cdot 41$  & 0  & 5  &  1  & 1  & 25 \\\hline
40455k1  & $3^2\cdot 5\cdot 29\cdot 31$  & 0  & 5  & 1  & 1  & 25\\\hline
52094m1  & $2\cdot 7\cdot 61^2$ &  0  & 5 &  1 &  1  & 25\\\hline
\end{tabular}
\end{center}
\end{table}

Thus if one believes
the second part of the Birch and Swinnerton-Dyer conjecture, then
if the hypothesis that N is prime
is dropped from the statement of  Theorem~\ref{thm:mainec}
(or of Theorem~\ref{thm:main}
and Corollary~\ref{cor:main}),
then the statement is no longer true. 
Among all the counterexamples that we found (up to level~130,000)
the curve 2366d1 has the smallest conductor, all except the curve 46683w1
have analytic rank~$0$, the curve 10621c1 has the smallest square-free conductor,
and the curves 33825o1,  40455k1, and 52094m1 are the only ones
for which $p \neq 3$. 
One might wonder if  the weaker statement
of Theorem~\ref{thm:mainec} that
if an odd prime~$p$ divides~$|E(\Q)_{\rm tor}|$, then 
$p$ does not divide~$|\Sha(E/\Q)|$ holds if $N$
is not prime.
If one assumes the second part of the Birch and Swinnerton-Dyer
conjecture, then  the curve
10621c1 (among others) shows that
this need not hold for $p=3$ 
(note also that for this curve,
level is square-free and not divisible by~$p$); however the weaker version
does hold in the examples for any prime $p>3$. The primes~$2$ and~$3$
are rather special from the point of view of component groups
(e.g., see Remark 2 on p.~175 of~\cite{mazur:eis}).
This raises the following question: 
\begin{que}
Is it true that 
if a prime~$p > 3$ divides~$|E(\Q)_{\rm tor}|$, then 
$p$ does not divide~$|\Sha(E/\Q)|$, perhaps under the 
restriction that $N$ is square-free?
\end{que}
While the data mentioned above does support an affirmative answer,
in the range of Cremona's database, if $p$ is a prime bigger than~$3$,
then the order of the torsion subgroup is often not divisible by~$p$
and the Birch and Swinnerton-Dyer
conjectural order of the Shafarevich-Tate group is rarely divisible by~$p$;
thus the chance of finding an example where both
are divisible by~$p$ is very slim. Hence the data is not enough
to make the conjecture that the answer to the above question is yes.

\section{Proof of Theorem~\ref{thm:main}}
\label{proof}

As mentioned earlier, our proof follows that 
of~\cite[III.3.6]{mazur:eis} (which proves 
the theorem above for the case where $A = J_0(N)$), 
with an extra input coming 
from~\cite{emerton:optimal}.
We also take the opportunity to give some details that
were skipped in~\cite[III.3.6]{mazur:eis}.

Let $\mm$ be a maximal ideal of~$\T$ 
containing~$\Eis$ and having odd residue characteristic~$p$.
Following~\cite[\S16]{mazur:eis}, an odd prime number $\ell \neq N$ 
is said to be {\em good} if $\ell$ is not a $p$-th power
modulo~$N$, and $\frac{\ell - 1}{2} \not\equiv 0 \bmod p$.
As mentioned in~\cite[p.125]{mazur:eis}, there are always
some good primes. 
Let $\eta = 1+ \ell - T_\ell$
for a good prime~$\ell$. 
By~\cite[II.16.6]{mazur:eis}, $\Eis \T_\mm$ is principal and  
generated by~$\eta$. 
We will show below that the map induced 
on~$\Sha(A)$ by the map~$\eta$ on~$A$ is injective,
i.e., $\Sha(A)[\eta]=0$.
Then since $\Eis \T_\mm \subseteq \mm \T_\mm$, 
we see that $\eta \in \mm \T_\mm$, and so $\Sha(A)[\mm] 
= \Sha(A)_\mm[\mm \T_\mm] = 0$, which
proves the theorem. 

We now turn to showing that
the map induced 
on~$\Sha(A)$ by the map~$\eta$ on~$A$ is injective.
Let $S = {\rm Spec} \Z$.
Let $\JJ$ denote the N\'eron model of~$J_0(N)$ over~$S$,
and $\Aa$ the N\'eron model of~$A$ over~$S$. 
We denote the map induced by~$\eta$ on~$\JJ$ by~$\nJ$  
on~$\Aa$ by~$\eta$ again. 
Let $\Delta$ denote a finite set of primes of~$\T$, not containing~$\mm$,
but containing all other primes in the support of the $\T$-modules
$(\ker \eta)(\Qbar)$ and $(\coker \eta)(\Qbar)$. We shall work in
the category of $\T$-modules modulo the category of $\T$-modules
whose supports lie in~$\Delta$. Thus, all equalities, injections, 
surjections, exact sequences,
etc., below shall mean ``modulo'' $\Delta$.
If $\GG$ is a group scheme over~$S$, then we shall denote 
the associated $fppf$~sheaf on~$S$ also by~$\GG$. All cohomology
groups below are for the $fppf$~topology over~$S$.

Now $\Sha(A)$ is a submodule of~$H^1(S,\Aa)$
by the Appendix of~\cite{mazur:tower}. We will show below
that $\eta$ induces an injection on~$H^1(S,\Aa)$. 
Then the map on~$\Sha(A)$ induced by~$\eta$ is also injective,
as was to be shown. It remains to prove the claim
that $\eta$ induces an injection on~$H^1(S,\Aa)$, which is what
we do next. 

Let $C$ denote the finite flat subgroup scheme of~$\JJ$ generated
by the subgroup of~$J_0(N)(\Q)$ generated by the divisor~$(0)-(\infty)$;
it is called the cuspidal subgroup in~\cite[\S~II.11]{mazur:eis}.
Let $\Sigma$ denote the 
the Shimura
subgroup of~$\JJ$ as in~\cite[\S~II.11]{mazur:eis}; 
by~\cite[II.11.6]{mazur:eis}, it is 
a $\mu$-type group (i.e., 
a finite flat group scheme whose Cartier dual is a 
constant group).
Let $C_p$ denote the $p$-primary component of the cuspidal
subgroup~$C$ and 
$\Sigma_p$ the $p$-primary component of 
the Shimura subgroup~$\Sigma$.
As shown in the proof of Lemma~16.10 in Chapter~II of~\cite{mazur:eis},
$\nJ$ is an isogeny.
By~\cite[II.16.6]{mazur:eis}, 
$\ker \nJ = C_p \oplus \Sigma_p$ 
as group schemes, and hence as $fppf$~sheaves.
By Theorem~4(i) in~\cite[\S7.5]{neronmodels}, the map
$\Aa \ra \JJ$ is a closed immersion. 
Thus $\ker \eta = (\Aa \cap C_p) \oplus (\Aa \cap \Sigma_p)$. 
Let $\CA$ denote the group $H^0(S, \Aa \cap C_p) = A \cap C_p$.
Then, since $\Sigma$ is a $\mu$-type subgroup, we have
\begin{eqnarray} \label{eqn:h0}
H^0(S,\ker \eta) = \CA.
\end{eqnarray}
Also, since $C_p$ is a constant group scheme
and $\Sigma_p$ is a $\mu$-type group scheme,
$C_p \cap \Aa$ and $\Sigma_p \cap \Aa$ are admissible group schemes
in the sense
of \S~I.1(f) of~\cite{mazur:eis}, and
so by~\cite[I.1.7]{mazur:eis}, $H^1(S,C_p \cap \Aa) = 0$
and $H^1(S,\Sigma_p \cap \Aa) = 0$
(in the notation of~\cite[I.1.7]{mazur:eis}: for $G = C_p \cap \Aa$,
$h^0(G) = \alpha(G)$, and $\delta(G)=0$ since $G$ is finite, so $h^1(G) = 0$;
for $G = \Sigma_p \cap \Aa$, $h^0(G) = 0$, $\alpha(G) = 0$, and
$\delta(G) = 0$ since $G$ is finite by~\cite[II.11.6]{mazur:eis},
and so $h^1(G) = 0$).
Thus
\begin{eqnarray} \label{eqn:h1}
H^1(S,\ker \eta) = 0.
\end{eqnarray}

Consider the short exact sequence:
\begin{eqnarray} \label{eqn:ses1}
0 \ra \ker \eta \ra \Aa \stackrel{\eta}{\ra} \im \eta \ra 0.
\end{eqnarray}
Its associated long exact sequence is:
$$0 \ra H^0(S,\ker \eta) \ra H^0(S,\Aa) \stackrel{\eta}{\ra} H^0(S,\im \eta)
 \ra H^1(S,\ker \eta) \ra  \ldots.$$
Now $H^1(S,\ker \eta) = 0$ by~(\ref{eqn:h1}),
so we have an exact sequence:
\begin{eqnarray} \label{eqn:ses2}
0 \ra H^0(S,\ker \eta) \ra H^0(S,\Aa) \stackrel{\eta}{\ra} H^0(S,\im \eta) \ra 0.
\end{eqnarray}
Now consider the short exact sequence:
$$0 \ra \im \eta \ra \Aa \ra \coker \eta \ra 0.$$
Its associated long exact sequence is:
\begin{eqnarray}\label{eqn:les2}
0 &\ra & H^0(S,\im \eta) \ra H^0(S,\Aa) \stackrel{i}{\ra} H^0(S,\coker \eta)
 \ra \nonumber \\
& \ra & H^1(S,\im \eta) \stackrel{i'}{\ra} H^1(S,\Aa) \ra \ldots,
\end{eqnarray}
where $i$ and~$i'$ denote the indicated induced maps.
Combining~(\ref{eqn:ses2}) and~(\ref{eqn:les2}), we get the exact sequence
\begin{eqnarray} \label{eqn:combles}
0 & \ra & H^0(S,\ker \eta) \ra H^0(S,\Aa) \stackrel{\eta}{\ra} 
H^0(S,\Aa) \stackrel{i}{\ra} H^0(S,\coker \eta) \ra \nonumber \\
 & \ra  &H^1(S,\im \eta) \stackrel{i'}{\ra} H^1(S,\Aa) \ra \ldots.
\end{eqnarray}

\noi {\em Claim:} The map~$i$ is surjective.
\begin{proof}
Let $\Aa^0$ denote the identity component of~$\Aa$ and let
$\Phi$ denote the quotient (i.e., the component group):
\begin{eqnarray} \label{eqn:conn}
0 \ra \Aa^0 \ra \Aa \ra \Phi \ra 0.
\end{eqnarray}
By~\cite[Thm.~4.12(ii) and~(iii)]{emerton:optimal}, $\Phi$ is a constant 
group scheme and the specialization map $\CA \ra \Phi$ is an isomorphism.
Hence $\Phi$ is killed by~$\eta$. Also, 
by the argument in~\cite[\S~2.2]{grothendieck:monodromie}
(cf. Prop.~C.8 and Cor.~C.9 of~\cite{milne:duality}), $\eta$ is surjective
on~$\Aa^0$ (considered as an $fppf$~sheaf). 
In view of this, 
if we apply the map~$\eta$ to the
short exact sequence~(\ref{eqn:conn})  of $fppf$~sheaves
and consider the snake lemma, then we see
that $\coker \eta = \Phi$.
Now $H^0(S,\Aa) = A(\Q)$ is a finitely generated abelian group, and 
by the exactness of~(\ref{eqn:combles}), the kernel of
the map induced by~$\eta$ on this group
contains the finite group~$H^0(S,\ker \eta) = \CA$. 
Hence the cokernel of the map induced by~$\eta$ on~$H^0(S,\Aa)$
has order at least that of~$\CA$. 
But by the exactness of~(\ref{eqn:combles}), 
this cokernel is isomorphic
to the image of~$i$, which is contained in~$H^0(S,\coker \eta) = H^0(S,\Phi) 
= \CA$ (from the discussion above), and hence has order
at most that of~$\CA$.
Thus $i$ must be surjective, as was to be shown.
\end{proof}

\comment{
How use fact that $H^0(S,\ker \eta) = C_p$ and
$H^0(S,\coker \eta) = C_p$. Then since $C_p$ is a direct 
factor of~$H^0(S,\Aa)$, $H^0(S,\Aa) = C_p \oplus G$. 
Then $i(H^0(S,\Aa)) \isom H^0(S,\Aa)/\im \eta \isom (C_p \oplus G)/(0 \oplus G)
\isom C_p = H^0(S,\Phi)$. 
But considering that multiplication by~$\eta$ is surjective on~$H^0(S,\Aa^0)$
and trivial on~$\Phi$ (modulo~$\Delta$), the snake lemma shows 
that $H^0(S,\coker \eta)$ is a submodule of~$H^0(S,\Phi)$.
Thus $i$ is surjective. 
}

Using the claim above,
from the exactness of~(\ref{eqn:combles}), we see that
the map~$i'$ is injective.
Part of long exact sequence associated to~(\ref{eqn:ses1}) is
$$\ldots \ra H^1(S,\ker \eta) \ra H^1(S,\Aa) \stackrel{\eta}{\ra} H^1(S,\im \eta) \ra \ldots.$$
From this, using the fact that the natural map
$i': H^1(S,\im \eta) \ra H^1(S,\Aa)$  in~(\ref{eqn:combles})
is injective, 
we get the following exact sequence:
\begin{eqnarray} \label{eqn:eslast}
H^1(S,\ker \eta) \ra H^1(S,\Aa) \stackrel{\eta}{\ra} H^1(S,\Aa).
\end{eqnarray}
But $H^1(S,\ker \eta) = 0$ by~(\ref{eqn:h1}),
and so by~(\ref{eqn:eslast}),
$\eta$ induces an injection on~$H^1(S,\Aa)$,
as was left to be shown.

\bibliographystyle{amsalpha}         

\providecommand{\bysame}{\leavevmode\hbox to3em{\hrulefill}\thinspace}
\providecommand{\MR}{\relax\ifhmode\unskip\space\fi MR }
\providecommand{\MRhref}[2]{%
  \href{http://www.ams.org/mathscinet-getitem?mr=#1}{#2}
}
\providecommand{\href}[2]{#2}

\end{document}